\documentclass[12pt,a4paper,reqno,oneside]{amsart}
\usepackage{t1enc}
\usepackage{times}
\usepackage{amssymb}
\usepackage[mathscr]{euscript}
\usepackage{graphicx}
\usepackage{hyperref}
\usepackage{enumitem}
\usepackage[final]{changes}
\usepackage{bbm}
\colorlet{Changes@Color}{red}

\newtheorem{thm}{Theorem}[section]

\newtheorem{lem}[thm]{Lemma}
\newtheorem{cor}[thm]{Corollary}

\theoremstyle{remark}
\newtheorem{rem}[thm]{Remark}

\theoremstyle{definition}

\renewcommand{\phi}{\varphi} 
\newcommand{\sign}{\mathrm{sign}} 
\newcommand{\E}{\mathbb{E}} 

\renewcommand{\>}{\rangle}

\makeatletter
\DeclareRobustCommand\bfseries{%
	\not@math@alphabet\bfseries\mathbf
	\fontseries\bfdefault\selectfont
	\boldmath 
}
\makeatother
\begin{document}
	
	\date{Version of \today}
	\title{A sharp upper bound for the expected occupation density of It\^o processes with bounded irregular drift and diffusion coefficients}

		\begin{abstract}
 We find explicit and optimal upper bounds for the expected occupation density for an It\^o-process when its drift and diffusion coefficients are unknown under boundedness and ellipticity conditions on the coefficients. This is related to the optimal bound for the expected interval occupation found in \cite{ankirchner2021sharp}. In contrast, our bound is for a single point and the resulting formula is less involved. Our findings allow us to find explicit upper bounds for mean path integrals.
   \end{abstract}
\author[Kr\"uhner]{Paul Kr\"uhner}
\address[Paul Kr\"uhner]{Institute for Statistics and Mathematics, WU Wien}
\email[]{paul.eisenberg@wu.ac.at}
\urladdr{https://bach.wu.ac.at/d/research/ma/18422/}
\author[Xu]{Shijie Xu}
\address[Shijie Xu]{Institute for Financial and Actuarial Mathematics, University of Liverpool}
\email[]{ShijieXu@liverpool.ac.uk}




\maketitle
\pagestyle{plain}
\section{Introduction} 


     Let $X: \mathbb R_+\to\mathbb R$ be a measurable function. The occupation measure of $X$ up to time  $T\geq 0$ is $\mu_T(\Gamma)=m\{0\leq s\leq T :X_s\in \Gamma\}$, $m$ being Lebesgue measure and $\Gamma$ a Borel set on $\mathbb R$. It is the amount of time spent by $X$ in the set $\Gamma$ during $[0, T]$. $X$ has an occupation density on $[0,T]$ if $\mu_T$ is absolutely continuous with respect to the Lebesgue measure. Roughly speaking, $\mu_T(\Gamma)$ could be expressed as the sum of times spent by $X$ at each $y\in \Gamma $ during $[0,T]$ in the following sense
     $$\mu_T(\Gamma)=\int_{y\in \Gamma} \alpha_T(y)dy$$
     where $\alpha_T: \mathbb R\to \mathbb R$ is the occupation density. D.\ Geman and J.\ Horowitz \cite{geman1980occupation} provided a thorough survey for occupation densities for deterministic paths as well as occupation densities for Markov processes. P.\ Imkeller and D.\ Nualart \cite{imkeller1994integration} established a sufficient criterion for the existence of occupation density for stochastic processes, i.e.\ criteria that ensure that almost every path has occupation density. The occupation times formula \cite[Corollary VI 1.6]{revuz2013continuous} implies a connection between local times and occupation densities.

     In this paper we find explicit and sharp upper bounds for the expectation of the occupation density for elliptic It\^o-processes with bounded coefficients. To be precise, we like to find upper bounds for 
      $$ \rho_T(y):= \limsup_{N \rightarrow \infty}   \E\left[\frac{N}2\int_0^T 1_{\left\{|X_s-y|\leq \frac 1 N\right\}}ds\right] $$
     when $dX_t = \beta_t dt + \sigma_t dW_t$ is an $\mathbb R$-valued It\^o-process with bounded coefficients $\beta$, $\sigma$ with $\sigma \geq  a$ for some constant $a>0$. 
     
     We denote the set of pairs of $\mathbb R$-valued progressively measurable processes $(\beta, \sigma)$ with $\sigma_t\in [a,b]$ and $|\beta_t|\leq k\sigma^2_t$ for any $t\geq 0$ by $\mathcal A:=\mathcal A_{a,b,k}$ where $0<a\leq b$ and $0\leq k$. Note that if $k=0$ then the drift coefficient $\beta$ is zero. The case $a=b$ corresponds to the cases in which only constant diffusion coefficients $\sigma_t=a$ are considered.
      
      Let $\mathcal C_{\mathcal A}$ be the class of stochastic processes $X$ such that there is $(\beta, \sigma)\in \mathcal A$ with $$dX_t= \beta_tdt+\sigma_tdW_t,\quad t\in [0,\infty)$$ where $W$ is a standard Brownian motion and the starting point $X_0$ is deterministic. We are interested in the maximal expected occupation density at a given time $T\geq 0$ over all possible It\^o-processes $X\in \mathcal C_{\mathcal A}$:
     $$G(x,y,T):=\sup_{X\in \mathcal C_{\mathcal A}, X_0=x} \E[\alpha_T(y)]$$
     where $x=X_0$ is the starting point of $X$, $y\in \mathbb R$ is the level and $\alpha_T(y)$ is the occupation density of $X$ in level $y$ up to time $T$ if it exists. 
     However, under this condition, the existence of an occupation density is unknown to the best of our knowledge. Especially, a connection to the densities of $(X_t)_{t\geq 0}$ cannot be exploited because $X_t$ does not need to have a density, see E.\ Fabes and C.\ Kenig \cite[Theorem]{fabes1981examples}. Instead of using the direct definition above, we use the following "approximation" version. Let $G: \mathbb R \times \mathbb R \times \mathbb R_+ \to \mathbb R$ be defined by
     \begin{align}\label{eq:OG}
     G(x,y,T):=\sup_{X \in  \mathcal C_{\mathcal A}, X_0=x}\left(\limsup_{N \rightarrow \infty}  \E\left[\frac{N}2 \int_0^T 1_{\left\{|X_s-y|\leq \frac 1 N\right\}}ds\right]\right).
     \end{align}
     
In this paper, we derive an explicit formula for $G$. $G$ is the sharpest upper bound for processes in $\mathcal C_\mathcal A$ which means that given $x, y\in \mathbb R$ and $T\geq 0$, for any $\epsilon >0$, we can construct an It\^o-process starting at $x$ with specific drift and diffusion coefficients in $\mathcal A$ such that its expected occupation density with respect to the level $y$ at time $T$ is higher than $G(x,y,T)-\epsilon$, see the proof of Theorem \ref{t:main result 2} and Equation \eqref{X:conti}. 

Our main result is as follows:
\begin{thm}\label{thm:main}
   The function $G$ as defined in Equation \eqref{eq:OG} satisfies
$$G(x,y,T)= \int_{0}^{T} \left(\frac{b}{a^2\sqrt{t}}\phi(v(r,t))+\frac{b^2 k}{a^2} \Phi(v(r,t))\right)dt$$   
where $r:=|x-y|$ and  $v(r,t):= kb\sqrt{t}-\frac{r}{b\sqrt{t}}, t\geq 0, \phi(z):=\frac{1}{\sqrt{2\pi}}e^{-\frac{z^2}{2}}$ and $\Phi(z):=\frac{1}{\sqrt{2\pi}}\int_{-\infty}^{z}e^{-\frac{s^2}{2}}ds$ for any $x,y\in \mathbb R$ and $T\geq0$.

\end{thm}  

The purpose of the main result is to obtain disintegration bounds for expected time integrals. Also, see Corollary \ref{c:disintegration2} for an improved version.
\begin{cor}\label{c:disintegration}
  Let $dX_t = \beta_t dt + \sigma_t dW_t$ be a $1$-dimensional It\^o-process with $x:=X_0$ deterministic, $|\beta_t|\leq k \sigma_t^2$ and $\sigma_t\in[a,b]$ for any $t\geq 0$. Then the expected occupation density 
 $ \lambda^X_T(y) := \underset{N \rightarrow \infty}\limsup \  \E\left[\frac{N}2\int_0^T 1_{\left\{|X_s-y|\leq \frac 1 N\right\}}ds \right]$
of $X$ is bounded (in $y$) and fulfils
 $ \lambda^X_T(y) \leq G(x,y,T) $
 for any $T\geq 0$, $y\in\mathbb R$. Moreover, for measurable $f:\mathbb R\rightarrow[0,\infty]$ one has
  $$ \E\left[ \int_0^T f(X_s) ds\right] \leq \int_{-\infty}^\infty G(x,y,T)f(y) dy $$
  for any $T\geq 0$.
\end{cor}

In the special case $a=b$, where the diffusion coefficient $\sigma_t=a$ is constant, D.\ Ba\~nos and P.\ Kr\"uhner \cite{banos2016} and P.\ Kr\"uhner and S.\ Xu \cite{kruhner2023explicit} provided an upper bound $\beta(x,y,t)$ for the density of the marginals $X_t$. In this case, one has
 $$ G(x,y,T) \leq \int_0^T \beta(x,y,s) ds.$$
By inspecting the formulas one can see that this happens to be an equality. Obviously, there are numerous methods to find upper bounds for the density of the marginals but as it is shown in E.\ Fabes and C.\ Kenig \cite[Theorem]{fabes1981examples}, there are processes $X\in\mathcal C_{\mathcal A}$ which do not even possess density. Moreover, there is a more explicit example of a process $X\in\mathcal C_{\mathcal A}$ in J.\ McNamara \cite[Theorem 3]{mcnamara1985regularity} with unbounded density.

In order to prove Theorem \ref{thm:main} we first study its exponentially stopped variant and find its value in Theorem \ref{t:V=Q}. 

Apart from interval occupation bounds studied in S.\ Ankirchner and J.\ Wendt \cite{ankirchner2021sharp}, S.\ Ankirchner, C.\ Blanchet-Scalliet and M.\ Jeanblanc \cite{ankirchner2017controlling} studied exponential Brownian processes and their expected occupation. P.\ Kr\"uhner and J.\ Eisenberg \cite[Appendix A]{eisenbergmeasuring} found an occupation bound for stopped processes where the diffusion coefficient is fixed and the drift has asymmetric bounds where one bound is possibly infinite. They focused on an application of their bound.

This paper is arranged as follows. In Section \ref{s:Laplace}, we study $G$ through its (modified) Laplace transform $V_\lambda(x,y) = \int_0^\infty \lambda e^{-\lambda T}G(x,y,T) dT$. In order to identify $V_\lambda$ we start out by understanding $V_\lambda$ as the value function of some control problem \eqref{eq:LTP}. In section \ref{s:H} we use a verification argument paired with a self-improvement argument to show that the inverse Laplace transform of $V_\lambda$ is, indeed, $G$. In the appendix we have summarised a technical bound needed for the proofs and give an application to Theorem \ref{thm:main} which is summarised in Corollary \ref{c:disintegration2}.
     
\subsection{Notations}
We define the signum function $\mathrm{sign}: \mathbb R\to \{-1,+1\}$ by 
\begin{align*}
	\mathrm{sign}(x):=
	\begin{cases}
	1 \quad&\text{if} \quad x>0, \\
        -1 \quad&\text{if} \quad x\leq 0.
	\end{cases}
\end{align*}  
We denote the absolute value of $x\in\mathbb R$ by $|x|$. For an $\mathbb R$-valued stochastic process $X$ we denote its local time in the sense of \cite[Theorem VI.1.2]{revuz2013continuous} by $L_t^y(X), t\geq 0, y\in \mathbb R$ (if it exists) or simply by $L_t^y$ if it is clear which process is meant. We denote the Dirac measure concentrated on $y\in\mathbb R$ by $\delta_y$.
$\mathcal L[\cdot](\lambda)$ denotes the Laplace transform at position $\lambda >0$, i.e.\ for an integrable functions $f:\mathbb R_+\rightarrow\mathbb R$ we define
$$\mathcal L[f(\cdot)](\lambda):=\int_0^{\infty} f(t)e^{-\lambda t}dt, \quad \lambda >0. $$
For a real function $f$ which has right and left limits at $x\in\mathbb R$ we define the difference operator of $f$ at $x$ via
$$\Delta [f(\cdot)](x):= \lim_{\epsilon\downarrow 0}\left(f(x+\epsilon)-f(x-\epsilon) \right).$$ 
Note that $f$ does not need to be defined in $x$ for $\Delta [f(\cdot)](x)$ to be well defined. For multivariable functions $f:\mathbb R^n\rightarrow \mathbb R$ we denote the $k$-th
order (mixed) partial derivatives of $f$ with respect to $x_{j_1},\dots, x_{j_k}$ by $\partial_{x_{j_k}\dots x_{j_1}} f$, where $j_i\in \{1,\dots,n\}, i=1,\dots,k,$ and $k\in \mathbb N$. Further notations are used as in D.\ Revuz and M.\ Yor \cite{revuz2013continuous}.

\section{Exponentially stopped time control problem}\label{s:Laplace}
In order to solve the fixed-time optimisation problem \eqref{eq:OG}, we first solve the exponentially stopped-time optimisation problem. More precisely, we replace the fixed time $T\geq 0$ with an independent exponentially distributed time with rate $\lambda >0$. 
We proceed with this due to the following considerations: On the one hand side, exponentially stopped time is memoryless which simplifies the problem. On the other hand side, transforming into exponentially stopped time can be interpreted as a (modified) Laplace transform. The second consideration will be used in the next section via the inverse Laplace transform, which yields a candidate for the fixed-time problem.

We define $V_{\lambda}: \mathbb R \times \mathbb R \to \mathbb R$ for any $\lambda >0$ by 
\begin{align}\label{eq:LTP}
V_{\lambda}(x,y):=\sup_{(\beta,\sigma) \in \mathcal A, X_0=x} \limsup_{N \rightarrow  \infty} \frac{N}2 \E\left[\int_0^\infty \lambda e^{-\lambda t} 1_{\left\{\left|X^{\beta, \sigma}_t-y\right|\leq \frac 1 N\right\}}dt\right].
\end{align} 
This is a control problem and we believe that the corresponding HJB-equation is:
\begin{align}
-\lambda V_{\lambda}(x,y) +\sup_{(\beta,\sigma) \in U}\left\{ \partial_x V_{\lambda}(x,y)\beta + \frac 1 2 \partial_{xx} V_{\lambda}(x,y)\sigma^2\right\} &=0, \quad x\neq y\label{eq:HJMneq}
\\ \sup_{(\beta,\sigma) \in U}\left\{1_{\{y=x\}} + \frac 1 2 \Delta [\partial_x V_{\lambda}(\cdot,y)](x)\sigma^2\right\} &=0\label{eq:HJMneq1}
\end{align}
and 
\begin{align}\label{e:U}U:=\{(\beta,\sigma)\in\mathbb R^2:\sigma\in[a,b],|\beta|\leq k\sigma^2\}.\end{align} In fact, we first solve these equations in Lemma \ref{S:general} and show in Theorem \ref{t:V=Q} that the solution is indeed equal to $V_\lambda$.



Equation \eqref{eq:HJMneq} is a second-order ODE and Equation \eqref{eq:HJMneq1} is a boundary condition. This ODE with boundary condition has a unique bounded solution which we provide in the next lemma.
\begin{lem}\label{S:general}
Let $\lambda >0$ and $y\in \mathbb R$. We define the function $Q_{\lambda}(\cdot, y): \mathbb R\to \mathbb R$ by
$$Q_{\lambda}(x,y):=\frac{1}{\left(-k+\sqrt{k^2+\frac{2\lambda}{b^2} }\right)a^2}e^{\left(k-\sqrt{k^2+\frac{2\lambda}{b^2}}\right)|x-y|}
.$$ 
Then $Q_\lambda$ is bounded, $x\mapsto Q_\lambda(x,y)$ is absolutely continuous and a $C^2$-function outside $\{y\}$. It has a left-side $x$-derivative $\partial_x Q_\lambda(\cdot,y)$ everywhere which is given by
$$\partial_x Q_{\lambda}(x,y):=\frac{\sign(y-x)}{a^2}e^{\left(k-\sqrt{k^2+\frac{2\lambda}{b^2}}\right)|x-y|},\quad \lambda>0, x\in \mathbb R.$$ 
The left-side derivative $\partial_x Q_\lambda(\cdot,y)$ is of bounded variation with a single discontinuity in $\{y\}$ with $\Delta [\partial_x Q_\lambda(\cdot,y)](x) = -\frac{2}{a^2}1_{\{x=y\}}$. $Q_\lambda$ is a solution to Equation \eqref{eq:HJMneq}, \eqref{eq:HJMneq1}, i.e.\
\begin{align*}
-\lambda Q_{\lambda}(x,y) +\sup_{(\beta,\sigma) \in U}\left\{\partial_x Q_{\lambda}(x,y)\beta + \frac 1 2 \partial_{xx} Q_{\lambda}(x,y)\sigma^2\right\} &=0, \quad x\neq y
\\ \sup_{(\beta,\sigma) \in U}\left\{1_{\{y=x\}} + \frac 1 2 \Delta [\partial_x Q_{\lambda}(\cdot,y)](x)\sigma^2\right\} &=0
\end{align*}
for any $x\in\mathbb R$.
\end{lem}
\begin{proof}
The given function $Q_\lambda$ is obviously bounded, it is absolutely continuous (in $x$) and its left-side derivative $\partial_x Q_\lambda$ is given as above. Also, $Q_\lambda(\cdot,y)$ is $C^2$ on $\mathbb R\setminus\{y\}$ and $\Delta [\partial_x Q_\lambda(\cdot,y)](x) = -\frac{2}{a^2}1_{\{x=y\}}$ for any $x\in\mathbb R$. Plugging it into the left side of Equation \eqref{eq:HJMneq}, \eqref{eq:HJMneq1} yields the claim.
\end{proof}
 
The next theorem states that $Q_\lambda = V_\lambda$ where $Q_\lambda$ is from Lemma \ref{S:general} and $V_\lambda$ is from Equation \eqref{eq:LTP}. As an intermediate step we first show that $V_\lambda \leq Q_\lambda$.
\begin{lem}\label{lem:Q geq V} 
	Let $\lambda >0$, $y\in\mathbb R$ and $Q_{\lambda}: \mathbb R\times \mathbb R \to \mathbb R$ be defined as in Lemma \ref{S:general}. Then we have $$V_\lambda(x,y)\leq Q_\lambda (x,y)$$ for any $x\in \mathbb R$.
\end{lem}
\begin{proof}
	
	Without loss of generality we may assume that $y=0$ and for simplicity, we use the notation $Q_\lambda(x):=Q_\lambda(x,0)$ for any $x\in\mathbb R$.
 
	Let $X \in \mathcal C_{\mathcal A}$ and denote its drift coefficient by $\beta$ and its diffusion coefficient by $\sigma$, $X_0=x$. Lemma \ref{S:general} shows the requirement of It\^o-Tanaka's formula \cite[Theorem VI 1.5]{revuz2013continuous} which yields 
	\begin{align}
	Q_{\lambda}(X_t)&=Q_{\lambda}(X_0)+\int_{0}^{t}\partial_x Q_{\lambda}(X_s)dX_s+\frac{1}{2}\int_{\mathbb R}L_t^z \partial_{xx} Q_{\lambda}(dz)\nonumber\\
	&=Q_{\lambda}(X_0)+\int_{0}^{t}\partial_x Q_{\lambda}(X_s)\beta_sds+\int_{0}^{t}\partial_x Q_{\lambda}(X_s)\sigma_sdW_s\nonumber\\&\quad+\frac{1}{2}\left(\int_{0}^{t}\partial_{xx}Q_{\lambda}(X_s)\sigma^2_s 1_{\{X_s\neq 0\}}ds+\Delta [\partial_x Q_{\lambda}(\cdot)](0) L^0_t \right)\label{eq:Q:opt}\\
	&=Q_{\lambda}(X_0)+\int_{0}^{t}\left( \partial_x Q_{\lambda}(X_s)\beta_s+\frac{1}{2} \partial_{xx} Q_{\lambda}(X_s)\sigma^2_s 1_{\{X_s\neq 0\}}\right)ds-\frac{1}{a^2}L^0_t \nonumber\\&\quad+\int_{0}^{t}\partial_x Q_{\lambda}(X_s)\sigma_sdW_s\label{eq:Q:ITOTANAKA}
	\end{align}
where the first equation \eqref{eq:Q:opt} is due to the occupation times formula \cite[Corollary VI 1.6]{revuz2013continuous} and the last equation \eqref{eq:Q:ITOTANAKA} is due to Lemma \ref{S:general}.

Let $T$ be an exponentially distributed random time with intensity $\lambda>0$ which is independent of $X$. We get
\begin{align*}
	\E [Q_{\lambda}(X_T)] &=\int_{0}^{\infty} \E[Q_{\lambda}(X_t)] \lambda e^{-\lambda t}dt\\
	&=Q_{\lambda}(X_0)\nonumber\\
 &\quad+\int_{0}^{\infty} \int_{s}^{\infty} \lambda e^{-\lambda t}dt \E\left[\partial_x Q_{\lambda}(X_s)\beta_s+\frac{1}{2} \partial_{xx} Q_{\lambda}(X_s)\sigma^2_s 1_{\{X_s\neq 0\}}\right]ds\\&\quad-\frac{1}{a^2}\int_{0}^{\infty}\lambda e^{-\lambda s}\E [L^0_s]ds\\
	&=Q_{\lambda}(X_0)+\int_{0}^{\infty} e^{-\lambda s} \E\left[\partial_x Q_{\lambda}(X_s)\beta_s+\frac{1}{2} \partial_{xx} Q_{\lambda}(X_s)\sigma^2_s 1_{\{X_s\neq 0\}}\right] ds\\&\quad-\frac{1}{a^2}\int_{0}^{\infty}\lambda e^{-\lambda s}\E [L^0_s]ds.
\end{align*}
Using that $Q_\lambda$ is a solution to Equation \eqref{eq:HJMneq} yields
\begin{align*}
    \E [Q_{\lambda}(X_T)] &\leq Q_{\lambda}(X_0)+\int_{0}^{\infty} e^{-\lambda s} \E[\lambda Q_{\lambda}(X_s)]ds-\frac{1}{a^2}\int_{0}^{\infty}\lambda e^{-\lambda s}\E [L^0_s]ds\\
	&=Q_{\lambda}(X_0)+\E[Q_{\lambda}(X_T)]-\frac{1}{a^2}\E[L^0_T].
\end{align*}
We find
\begin{align*}
   \frac{1}{a^2} \E [L_T^0] \leq  Q_{\lambda}(X_0). 
\end{align*}

Repeating the arguments for arbitrary $y\in\mathbb R$ yields
\begin{align}\label{ineq:Q：L}
   \frac{1}{a^2} \E [L_T^y] \leq  Q_{\lambda}(X_0,y). 
\end{align}

We define $f_N(x):=\frac{N}2  1_{\left\{|x|\leq \frac{1}{N}\right\}}, x\in\mathbb R$ and apply the occupation times formula \cite[Corollary VI 1.6]{revuz2013continuous} to get
$$\frac{N}2\int_{0}^{T}1_{\left\{|X_s|\leq \frac{1}{N}\right\}}a^2ds \leq \frac{N}2\int_{0}^{T}1_{\left\{|X_s|\leq \frac{1}{N}\right\}}\sigma^2_sds=\frac{N}2\int_{-\frac 1 N}^{\frac 1 N}L^x_Tdx.$$
Hence, by applying expectation and using Equation \eqref{ineq:Q：L}, we get
$$\E\left[\frac{N}2\int_{0}^{T}1_{\left\{|X_s-y|\leq \frac{1}{N}\right\}}ds\right]\leq \frac{N}{2}\int_{y-\frac 1 N}^{y+\frac 1 N}\frac1{a^2}\E[L^z_T]dz \leq \frac{N}2 \int_{y-\frac1N}^{y+\frac1N} Q_\lambda(X_0,z)dz.$$
Since $z\mapsto Q_\lambda(X_0,z)$ is continuous we find due to the fundamental theorem of calculus that
$$\limsup_{N\rightarrow \infty} \E \left[\frac{N}2 \int_{0}^{T}1_{\left\{|X_s-y|\leq \frac{1}{N}\right\}} ds\right]\leq Q_\lambda(X_0,y) = Q_\lambda(x,y).$$
By the definition of $V_\lambda$, cf.\ Equation \eqref{eq:LTP} we find
 \begin{align*}   
 V_{\lambda}(x,y)&=\sup_{X \in \mathcal C_{\mathcal A}, X_0=x}\limsup_{N\rightarrow \infty} \E \left[\frac{N}2\int_{0}^{\infty} \lambda e^{-\lambda s}1_{\left\{|X_s-y|\leq \frac{1}{N}\right\}}ds\right]
 \\&\leq Q_{\lambda}(x,y)
 \end{align*}
 as claimed.
 \end{proof}

We now complete the proof that $Q_{\lambda}$ solves the exponentially stopped control problem \eqref{eq:LTP} by a self-improvement argument.
 
\begin{thm}\label{t:V=Q}
   The solution $Q_\lambda$ provided in Lemma \ref{S:general} and the function $V_\lambda$ from Equation \eqref{eq:LTP} are equal:
    $$V_\lambda = Q_\lambda $$
for any $\lambda > 0$.
\end{thm}

\begin{proof}
Lemma \ref{lem:Q geq V} yields $V_\lambda(x,y)\leq Q_\lambda (x,y)$ for any $x,y\in\mathbb R$ and $\lambda>0$.

For given $\lambda>0$ and $x,y\in\mathbb R$ we need to show that
$$Q_\lambda(x,y)\leq V_\lambda (x,y).$$

Without loss of generality, we may assume that $y=0$. For simplicity, we use the notation $Q_\lambda(x):=Q_\lambda(x,0)$ for any $x\in\mathbb R$. 
Fix $M\in\mathbb N_+$ and we construct a feedback control via choosing $$\sigma_M(x):= a+(b-a)g_M(|x|)$$ where $g_M\in C\left(\mathbb R_+,[0,1]\right)$ such that for any $x\in\mathbb R_+$
\begin{equation*}
g_M(x)=
\begin{cases}
0 & \text{if} \quad x\in [0, \frac1 M],\\
1 & \text{if} \quad x\in [\frac2 M,\infty].\\
\end{cases}  
\end{equation*}
 Let $X$ be the solution to the stochastic differential equation:
 $$ dX_t = -k \sigma_M(X_t)^2 \mathrm{sign}(X_t) dt + \sigma_M(X_t)dW_t, \quad X_0=x, t\geq 0$$
 and $T$ be an exponentially distributed random time with parameter $\lambda$ which is independent of $X$.
The control is now specified as:
 $$ \beta_t := -k \sigma_M(X_t)^2\mathrm{sign}(X_t), \quad \sigma_t := \sigma_M(X_t),\quad t\geq 0.$$
We have $(\beta,\sigma)\in\mathcal A$ and $X\in\mathcal C_{\mathcal A}$. We find that
\begin{align}
\E [Q_{\lambda}\left(X_T\right)] 
&=Q_{\lambda}(X_0)+\int_{0}^{\infty} e^{-\lambda s} \E\bigg[\left(\partial_x Q_{\lambda}\left(X_s\right)\beta_s+\frac{1}{2} \partial_{xx}Q_{\lambda}(X_s)\sigma_M(X_s)^2\right)\nonumber\\
&\quad\quad \cdot 1_{\left\{\frac 2 M\leq |X_s|\right\}}\bigg] ds\nonumber\\&\quad+\int_{0}^{\infty} e^{-\lambda s} \E\bigg[\left(\partial_x Q_{\lambda}(X_s)\beta_s+\frac{1}{2} \partial_{xx} Q_{\lambda}(X_s)\sigma_M(X_s)^2\right) \nonumber\\ &\quad\quad \cdot 1_{\left\{0<|X_s|<\frac 2 M\right\}}\bigg] ds\nonumber\\&\quad-\frac{1}{a^2}\int_{0}^{\infty}\lambda e^{-\lambda s}\E [L^0_s]ds \nonumber\\
&=Q_{\lambda}(X_0)+\int_{0}^{\infty} \lambda e^{-\lambda s} \E\left[Q_\lambda(X_s 
)  1_{\left\{|X_s|> 0\right\}}\right] ds\nonumber\\&\quad +\int_{0}^{\infty} e^{-\lambda s} \E\left[\bigg(\partial_x Q_{\lambda}(X_s)\left(\beta_s+kb^2\mathrm{sign}(X_s)\right)+\frac{1}{2}\partial_{xx} Q_{\lambda}(X_s) \right. \nonumber \\&\qquad\left. \cdot\left(\sigma_M(X_s)^2-b^2\right)\bigg) 1_{\left\{0<|X_s|< \frac 2 M\right\}}\right] ds\nonumber\\&\quad-\frac{1}{a^2}\int_{0}^{\infty}\lambda e^{-\lambda s}\E [L^0_s]ds\nonumber
\end{align}
where we used for the first inequality independence of $T$ and $X$ and \cite[Theorem VI. 1.5]{revuz2013continuous} and rearranged terms for the second equality.

There exists constant $C_1:=\sup_{|z|\in (0,2]}\left(|\partial_x Q_{\lambda}(z)|2kb^2+\frac{1}{2}\left|\partial_{xx} Q_{\lambda}(z)\right|b^2\right)$ does not depend on $M$ and we have
\begin{align}
&\quad\int_{0}^{\infty} e^{-\lambda s} \E\bigg[\left|\partial_x Q_{\lambda}(X_s)\left(\beta_s+kb^2\mathrm{sign}(X_s)\right)+\frac{1}{2} \partial_{xx} Q_{\lambda}(X_s)\left(\sigma_M(X_s)^2-b^2\right)\right| \nonumber\\ &\quad\quad \cdot 1_{\left\{0<|X_s|< \frac 2 M\right\}}\bigg] ds\nonumber\\
&\leq \int_{0}^{\infty}  e^{-\lambda s} \sup_{|z|\in (0,2]}\left(|\partial_x Q_{\lambda}(z)|2kb^2+\frac{1}{2}\left| \partial_{xx} Q_{\lambda}(z)\right|b^2\right) \mathbb P\left( 0<\left|X_s\right|< \frac 2 M\right) ds\nonumber\\
&\leq \int_{0}^{\infty} e^{-\lambda s} C_1 \mathbb P\left(|X_s|< \frac 2 M\right)ds\nonumber
=: \epsilon_M. \nonumber
\end{align} 

Thus, we have
\begin{align*}
    Q_{\lambda}(x) \leq \frac{1}{a^2} \E [L^0_T]+\epsilon_M.
\end{align*}
Applying Lemma \ref{lem:Q geq V} to $\epsilon_M$ we find
\begin{align*}
    \epsilon_M&\leq C_1\int_{-2/M}^{2/M}Q_{\lambda}(x)dx\leq \frac{4C_1 Q_{\lambda}(0)}{M}.
\end{align*}
Note that $C_2:=4C_1Q_{\lambda}(0)$ does not depend on the specific choice of $M$ so we have
$$ \epsilon_M \leq \frac{C_2}{M}.$$

We find that 
\begin{align*}
  Q_{\lambda}(x) &\leq \frac{1}{a^2}  \E[L^0_T] + \frac{C_2}{M} \\
             &= \frac1{a^2} \left(\limsup_{N \rightarrow \infty}  \frac N2 \E\left[\int_0^\infty \lambda e^{-\lambda s} 1_{\left\{|X_s|\leq \frac 1 N\right\}}d\<X,X\>_s\right]\right) + \frac{C_2}{M} \\
             &\leq \sup_{Z \in \mathcal C_{\mathcal A}, Z_0=x}\left(\limsup_{N \rightarrow \infty}  \frac N2 \E\left[\int_0^\infty \lambda e^{-\lambda s} 1_{\left\{|Z_s|\leq \frac 1 N\right\}}ds\right]\right) + \frac{C_2}{M} \\
             &= V_\lambda(x) + \frac{C_2}{M}
\end{align*}
where we used \cite[Corollary VI.1.9]{revuz2013continuous} for the first equality.
This is true for all $M\in\mathbb N_+$ and consequently, $V_{\lambda}(x)=Q_{\lambda}(x)$ for any $\lambda>0, x\in\mathbb R$ as required.
\end{proof}

\section{Fixed time horizon optimization problem}\label{s:H}
 Throughout this section, we choose $y\in\mathbb R$ to be fixed. Define $r:=|x-y|$ which is the distance between the starting point $x$ and the level $y$ whenever $x\in\mathbb R$ is given. For this section, we write $V_{\lambda}(r):=V_{\lambda}(x,y)$. We would like to find the inverse Laplace transform $H_T(r)$ of $\lambda\rightarrow V_{\lambda}(r)/\lambda$ at position $T\geq 0$, which means we want to find some function $H:\mathbb R_+\times \mathbb R_+\rightarrow \mathbb R$ which satisfies the following equation:
$$V_{\lambda}(r)=\int_{0}^{\infty} \lambda e^{-\lambda T}H_T(r)dT, \quad r\in \mathbb R_+, \lambda>0.$$

\begin{lem}\label{H:explicit}
	The inverse Laplace transform of $\lambda\rightarrow \frac{V_{\lambda}(r)}{\lambda}$ is given by
\begin{align}\label{eq:H}
		H_T(r)=&\int_{0}^{T} \left(\frac{b}{a^2\sqrt{t}}\phi(v(r,t))+\frac{b^2 k}{a^2} \Phi(v(r,t))\right)dt, \quad T> 0, r\geq 0 
\end{align}
		where $v(r,t):= kb\sqrt{t}-\frac{r}{b\sqrt{t}}$, $\phi(z):=\frac{1}{\sqrt{2\pi}}e^{-\frac{z^2}{2}}$ and $\Phi(z):=\frac{1}{\sqrt{2\pi}}\int_{-\infty}^{z}e^{-\frac{s^2}{2}}ds$.
\end{lem}
\begin{proof}
Using \cite[Appendix 3, Eq.\ (a), (d) \& (10)]{borodin2015handbook}, we find that the Laplace transform of $T\mapsto H_T$ given by Equation \eqref{eq:H} satisfies
$$\lambda\mathcal{L}[H_{\cdot}(r)](\lambda)=V_{\lambda}(r)$$
for any $\lambda>0, r\geq 0$.
%
\end{proof}

Next, we gather some properties of $H_T$, which are given in Lemma \ref{H:explicit}. We find that $r\mapsto H_T(r)$ is positive, decreasing, and convex and we identify the right-handed limit of its derivative $\partial_r H_T(r)$ at $r=0$ is $-\frac{1}{a^2}$. 
These will be needed for verification arguments later in Lemma \ref{HJB:H} and Lemma \ref{Ito-Tanaka}. 
\begin{lem}\label{mono}
	Let $H$ be as in Lemma \ref{H:explicit}. For any $r> 0$ and $T\geq 0$, we have 
	\begin{align}
		 {\partial_r H_T(r)}&\leq 0\leq {\partial_{rr} H_T (r)}\nonumber\nonumber		
	\end{align}
	and for any $T>0$ we also have
	\begin{align}\label{first der}
		\lim_{r\downarrow 0}{\partial_r  H_{T}(r)}=-\frac{1}{a^2}.
	\end{align}
\end{lem}
\begin{proof}
Taking the first derivative of $H_T$ with respect to $r$ we get
\begin{align*}
    {\partial_r H_T(r)}=\int_{0}^{T} \frac{1}{a^2\sqrt{2\pi t^3}}e^{-\left( \frac{r}{\sqrt{2b^2 t}}-k\sqrt{\frac{b^2 t}{2}}\right)^2}\left(-\frac{r}{b}\right)dt, \quad T\geq 0, r>0.
\end{align*}
It is clear that ${\partial_r H_T(r)}\leq 0$ for any $T\geq 0, r\geq 0$ because the integrand is non-positive. 

Taking the second derivative of $H_T$ with respect to $r$ we get
\begin{align*}
    {\partial_{rr} H_T (r)}=\int_{0}^{T}\frac{1}{a^2 b \sqrt{2\pi t^3}}e^{-\left( \frac{r}{\sqrt{2b^2 t}}-k\sqrt{\frac{b^2 t}{2}}\right)^2}\left(\frac{r^2}{b^2 t}-kr -1\right)dt, \quad T\geq 0, r>0.
\end{align*} 
Recall that for any $\lambda>0$, $V_{\lambda}$ and its second derivative with respect to $r$ are
\begin{align*}
	V_{\lambda}(r)&=\int_0^{\infty} \lambda e^{-\lambda T} H_{T}(r)dT,& \quad &&r\geq 0\\
	{\partial_{rr} V_{\lambda}(r)}&=\int_0^{\infty} \lambda e^{-\lambda T} {\partial_{rr} H_{T}(r)}dT
	= \int_{0}^{\infty} e^{-\lambda T} {\partial_{Trr}  H_{T}(r)} dT, &&&\quad r>0
\end{align*}
where we used integration by parts for the second equality.

Let $\lambda \downarrow 0$ and take the limit for both sides, since $H_0(r)=0$, for any $r>0$, we have
\begin{align*}
\lim_{T\to \infty}{\partial_{rr}  H_{T}(r)}&=\lim_{T\to \infty}
{\partial_{rr}  H_{T}(r)}-{\partial_{rr}  H_{0}(r)}
\\&=\int_0^{\infty} \left(\lim_{\lambda\to 0} e^{-\lambda T}\right) {\partial_{Trr}  H_{T}(r)} dT\\
&=\lim_{\lambda\downarrow 0}{\partial_{rr} V_{\lambda}(r)}\\&=\lim_{\lambda\downarrow 0}\frac{\sqrt{k^2+\frac{2\lambda}{b^2}}-k}{a^2} e^{\left(k-\sqrt{k^2+\frac{2\lambda}{b^2}}\right)r} \\&=0.
\end{align*}

Observe that for any fixed $r> 0$
\begin{align*}
{\partial_{rr} H_{\cdot}(r)}: \mathbb R_+\to \mathbb R, T\mapsto \int_{0}^{T}\frac{1}{a^2 b \sqrt{2\pi t^3}}e^{-\left( \frac{r}{\sqrt{2b^2 t}}-k\sqrt{\frac{b^2 t}{2}}\right)^2}\left(\frac{r^2}{b^2 t}-kr -1\right)dt
\end{align*} 
is increasing on $\left[0,\frac{r^2}{(kr+1)b^2}\right]$ and decreasing on $ \left(\frac{r^2}{(kr+1)b^2}, \infty\right)$ with 
$${\partial_{rr} H_{0}(r)}=0, \quad \lim_{T\to \infty}{\partial_{rr}  H_{T}(r)}=0$$ which implies ${\partial_{rr}  H_{T}(r)}\geq 0$ for all $T\geq 0$ and $r>0$.

Also, we can repeat these steps for first order derivative $\frac{\partial H_T(r)}{\partial r}$ with respect to $r$ for any $r>0$ and $T>0$,
\begin{align*}
	{\partial_{r} V_{\lambda}(r)}&=\int_0^{\infty} \lambda e^{-\lambda T} {\partial_r H_T(r)}dT\\
	&=\int_{0}^{\infty} e^{-\lambda T}{\partial_{Tr} H_{T}(r)}dT
\end{align*}
where we used integration by parts.
Let $\lambda \downarrow 0$ we find
$$\lim_{T\to \infty}{\partial_r  H_{T}(r)}=\lim_{\lambda\downarrow 0}	{\partial_r V_{\lambda}(r)}=\lim_{\lambda\downarrow 0}-\frac{1}{a^2} e^{\left(k-\sqrt{k^2+\frac{2\lambda}{b^2}}\right)r} =-\frac{1}{a^2}.$$

For any fixed $T>0$. When $r\downarrow 0$ we get
\begin{align*}
	\lim_{r\downarrow 0}{\partial_{Tr} H_T(r)}&=\lim_{r\downarrow 0}\frac{1}{a^2\sqrt{2\pi T^3}}e^{-\left( \frac{r}{\sqrt{2b^2 T}}-k\sqrt{\frac{b^2 T}{2}}\right)^2}\left(-\frac{r}{b}\right)=0
\end{align*}
which means that $\lim_{r\downarrow 0}{\partial_{r} H_T(r)}$ is constant in $T$. Hence, we have 
$$\lim_{r\downarrow 0}{\partial_r H_T(r)}=\lim_{r\downarrow 0}\lim_{T\to \infty}{\partial_r H_{T}(r)}=-\frac{1}{a^2}, \quad T> 0.$$
\end{proof}

We believe the HJB-equation belonging to the original control problem \eqref{eq:OG} is 
\begin{align}
-{\partial_t G(x,y,t)}+\sup_{(\beta,\sigma) \in U}\left\{ {\partial_x G(x,y,t)}\beta+\frac12 {\partial_{xx} G(x,y,t)} \sigma^2 \right\}&=0, \quad x\neq y\label{eq:HJB1}\\
\sup_{(\beta,\sigma) \in U}\left\{1_{\{x=y\}}+\frac{\sigma^2}2  \Delta\left[{\partial_x G(\cdot,y,t)}\right](x) \right\}&=0\label{eq:HJB2}\\
G(x,y,0)&=0 \label{eq:HJB3}
\end{align}
for any $x,y\in \mathbb R$ where $U$ is defined in Equation \eqref{e:U}. 

 Equation \eqref{eq:HJB1} is a PDE with boundary condition Equation \eqref{eq:HJB3} and pasting condition \eqref{eq:HJB2}. Instead of solving the PDE directly, we use that $V_{\lambda}$ is the Laplace transform of $H_T$ and we prove that $G$ given by $G(x,y,t):=H_t(|x-y|)$ solves the equations above.
\begin{lem}\label{HJB:H}
	Let $H_T: \mathbb R\times \mathbb R\to \mathbb R_+$ be the inverse Laplace transform of $V_{\lambda}$ provided in Lemma \ref{H:explicit} and define $G(x,y,t):=H_t(|x-y|)$ for $t\geq 0$, $x,y\in\mathbb R$. Then this $G$ solves Equations \eqref{eq:HJB1}, \eqref{eq:HJB2} and \eqref{eq:HJB3}.
	\end{lem}
\begin{proof}
Lemma \ref{H:explicit} yields that $H$ satisfies Equation \eqref{eq:HJB3}.
  
	 We recall that Theorem \ref{t:V=Q} yields
	\begin{align}
	-\lambda Q_{\lambda}(x,y) +\sup_{(\beta,\sigma) \in U}\left\{\partial_x Q_{\lambda}(x,y)\beta + \frac 1 2 \partial_{xx} Q_{\lambda}(x,y)\sigma^2\right\} = 0, \quad x\neq y\label{eq:LT:1}
	\end{align}
 for any $x,y\in\mathbb R$ with $x\neq y$ and $\lambda >0$ where the supremum is attained in:
	\begin{equation*}
	\left(\beta(x),\sigma(x)\right)=
	\begin{cases}
	(-kb^2, b) & \text{if} \quad x>y,\\
	(0, a) & \text{if} \quad x=y,\\
	(kb^2, b) & \text{if} \quad x<y.\\
	\end{cases}      
	\end{equation*}
 Since $\lambda\mathcal{L}[H_{\cdot}(r)](\lambda)=Q_{\lambda}(r)$, and we have $\mathcal L[\partial_{\cdot} H_{\cdot}(r)](\lambda)=\lambda\mathcal L[H_{\cdot}(r)](\lambda)$, $\mathcal{L}[\partial_r H_{\cdot}(r)](\lambda)=\partial_r\mathcal{L}[ H_{\cdot}(r)](\lambda)$ and $\mathcal{L}[\partial_{rr} H_{\cdot}(r)](\lambda)=\partial_{rr}\mathcal{L}[ H_{\cdot}(r)](\lambda)$ for $\lambda>0$ and $r\geq 0$, taking inverse Laplace transform on Equation \eqref{eq:LT:1} yields with our convention $r=|x-y|$ that
  $$ -{\partial_t H_t(r)}+\left\{ {\partial_r H_t(r)}\beta(x)+\frac12 {\partial_{rr} H_t(r)} \sigma(x)^2 \right\}=0, \quad t>0.$$
    Lemma \ref{mono} yields that
     $$ {\partial_r H_t(r)}\beta(x)+\frac12 {\partial_{rr} H_t(r)} \sigma(x)^2 = \sup_{(\beta,\sigma)\in U}\left\{{\partial_r H_t(r)}\beta+\frac12 {\partial_{rr} H_t(r)} \sigma^2\right\}$$
   and, hence, $G$ solves Equation \eqref{eq:HJB1}. 
 
	That $G$ solves Equation \eqref{eq:HJB2} follows immediately from the last equality shown in Lemma \ref{mono}. Lemma \ref{H:explicit}  yields that $G$ satisfies Equation \eqref{eq:HJB3}.
\end{proof}

Before we prove our main result, we prove that a time-inhomogeneous variation of the It\^o-Tanaka's formula holds in a specific setup.
\begin{lem}\label{Ito-Tanaka}
	Let $H$ be given as in Lemma \ref{H:explicit} and define $f(t,x):= H_{T-t}(|x|)$, $0\leq t\leq T$ and $x\in\mathbb R$. Let $X\in\mathcal C_{\mathcal A}$ with $X_0$ deterministic and denote its drift and diffusion coefficients by $\beta$ and $\sigma$ where $(\beta, \sigma)\in \mathcal A$.
 Then we have
	\begin{align*}
	f(t,X_t) &=f(0,X_0)+\int_0^t \partial_t f(s,X_s)ds+\int_0^t \partial_x f(s,X_s)dX_s+\frac{1}{2}\int_0^t \partial_{xx}f(s,X_s)\sigma^2_sds\nonumber\\&\quad -\frac{1}{a^2}L^0_t\nonumber
	\end{align*}
 for any $t\in [0,T)$.
\end{lem}
\begin{proof}
 
	Applying It\^o-Tanaka's formula \cite[Theorem VI 1.6]{revuz2013continuous} for the function $x\mapsto f(t,x)$ when $t$ is fixed we find
	\begin{align}\label{eq:f}
	f(t,X_t)&=  f(t,X_0)+ \int_{0}^{t} \partial_x f(t,X_s)dX_s+\frac{1}{2}\int_{\mathbb R} L^z_t \partial_{xx} f(t,dz) \nonumber \\
	&= f(0,X_0)+\int_0^{t} \partial_t f(s,X_0)ds+\int_{0}^{t} \left(\partial_x f(s,X_s)+\int_{s}^{t}\partial_{tx}f(u,X_s)du\right)dX_s\nonumber \\
	&\quad+\frac12\left(\int_\mathbb{R}L^z_t \partial_{xx}f(t,z)dz+\int_{\mathbb R}L^0_t\Delta [\partial_x f(t,\cdot)](x)\delta_0(dx)\right)\nonumber \\
	&=f(0,X_0)+\int_0^{t} \partial_t f(s,X_0)ds+\int_{0}^{t} \left(\partial_x f(s,X_s)+\int_{s}^{t}\partial_{tx}f(u,X_s)du\right)dX_s\nonumber \\
	&\quad+\frac12 \int_0^t \partial_{xx} f (t,X_s)\sigma_s^2ds-\frac{1}{a^2}L^0_t
	\end{align}
	where the last equation is from:
        \begin{align*}
            \Delta [\partial_x f(t,\cdot)](0)=\partial_x f(t,0^+)-\partial_x f(t,0^-)=-\frac{2}{a^2}
        \end{align*}
        as Equation \eqref{first der} implies $\partial_x f(t,0^+)=-\frac{1}{a^2}$ and we have $\partial_x f(t,0^-)=\frac{1}{a^2}$ by symmetry.
        
        Since $x \mapsto \partial_t f(s, x) $ is a $C^2$-function, using It\^o's formula we have
	\begin{align*}
	\partial_t f(s,X_s)=\partial_t f(s,X_0)+\int_0 ^s \partial_{xt}f(s,X_u)dX_u+\frac12 \int_0^s \partial_{xxt}f (s,X_u)\sigma_u^2du.    
	\end{align*}
	Applying Fubini's theorem yields
	\begin{align*}
	\int_{0}^{t}\partial _t f(s,X_s)ds&=\int_{0}^{t} \partial_t f(s, X_0)ds+ \int_{0}^{t}\int_{0}^{s}\partial_{xt}f(s, X_u)dX_uds\\&\quad+\frac12\int_{0}^{t}\int_{0}^{s} \partial_{xxt}f(s,X_u)\sigma_u^2duds \\ 
	&= \int_{0}^{t} \partial_t f(s, X_0)ds+ \int_{0}^{t}\int_{s}^{t}\partial_{xt}f(u, X_s)dudX_s\\&\quad+\frac12\int_{0}^{t}\int_{u}^{t} \partial_{xxt}f (s,X_u)\sigma_u^2dsdu \\
        &= \int_{0}^{t} \partial_t f(s, X_0)ds+ \int_{0}^{t}\int_{s}^{t}\partial_{xt}f(u, X_s)dudX_s\\&\quad+\frac12\int_{0}^{t}(\partial_{xx}f(t,X_u)-\partial_{xx}f(u,X_u))\sigma_u^2du. \\
	\end{align*}
	Substituting 
 \begin{align*}
     &\quad\int_0^t \partial_t f(s,X_0)ds+\int_{0}^{t}\int_{s}^{t}\partial_{tx}f(u, X_s)dudX_s\\&= \int_0^t \partial_t f(s,X_s)ds-\frac12\int_{0}^{t}\left(\partial_{xx} f(t,X_u)-\partial_{xx}f(u,X_u)\right)\sigma_u^2du
 \end{align*}
 into Equation \eqref{eq:f}, we find 
	\begin{align*}
	    f(t,X_t)&=f(0,X_0)+\int_0^t \partial_t f(s,X_s)ds+\int_0^t \partial_x f(s,X_s)dX_s+\frac{1}{2}\int_0^t \partial_{xx}f(s,X_s)\sigma^2_sds
	\\&\quad-\frac{1}{a^2}L^0_t.
	\end{align*}	
\end{proof}

We proceed to show that $G=H$ where $G$ is given in Problem \eqref{eq:OG}. Like in the exponentially stopped case we first show that $G\leq H$ and then use a self-improvement technique in the proof of Theorem \ref{t:main result 2} below to show that indeed $G=H$.
\begin{lem}\label{lem:leq}	For the function $G$ given in Problem \eqref{eq:OG} we have  $G(x,y,T) \leq H_{T}(|x-y|)$ for any $T\geq 0$ and $x,y\in\mathbb R$ where $H$ is defined in Lemma \ref{H:explicit}.
\end{lem} 
\begin{proof}	
 Without loss of generality we may assume that $y=0$. Let $X\in\mathcal C_{\mathcal A}$ be any controlled process, i.e.\ $dX_t = \beta_t dt + \sigma_t dW_t$ with $X_0=x$, $t\geq 0$ and $(\beta,\sigma)\in\mathcal A$. By Lemma \ref{Ito-Tanaka} we have
	\begin{align*}
		H_{T-t}(|X_t|)&=H_T(|X_0|)-\int_{0}^{t}\partial_{T} H_{T-s}(|X_s|)ds+\int_{0}^{t} \partial_r H_{T-s}(|X_s|)\sign(X_s)dX_s\\
		&\quad+\frac{1}{2}\int_{0}^{t}\partial_{rr}H_{T-s}(|X_s|)\sigma^2_sds-\frac{1}{a^2}L^0_t\\
	            &=H_T(|X_0|)+\int_{0}^{t}\Big(-\partial_{T} H_{T-s}(|X_s|)+ \sign(X_s) \partial_r H_{T-s}(|X_s|)\beta_s\\&\quad +\frac{1}{2}\partial_{rr} H_{T-s}(|X_s|)\sigma^2_s\Big)1_{\{X_s\neq 0\}}ds-\frac{1}{a^2}L^0_t +\int_{0}^{t}\partial_r H_{T-s}(|X_s|)\sigma_s dW_s.
	\end{align*}
Taking expectations on both sides, we find
\begin{align}\label{eq:E neq}
	\E[H_{T-t}(|X_t|)]&=H_T(|X_0|)\nonumber\\
	&\quad+\E\left[\int_{0}^{t}\Big(-\partial_{T} H_{T-s}(|X_s|)+ \sign(X_s)\partial_r H_{T-s}(|X_s|)\beta_s\right. \nonumber\\
 &\quad \left.+\frac{1}{2}\partial_{rr}H_{T-s}(|X_s|)\sigma^2_s\Big)1_{\{X_s\neq 0\}}ds\right]
-\E\left[\frac{1}{a^2}L^0_t\right]\\
&\leq  H_T(|X_0|)-\frac{1}{a^2} \E[L^0_t]\nonumber
\end{align}
where the inequality holds because $H$ is a solution to Equation \eqref{eq:HJB1} according to Lemma \ref{HJB:H}. Pushing $t$ to $T$ in Inequality \eqref{eq:E neq} yields
\begin{align*}
0=\E[H_0(|X_T|)]\leq H_T(|X_0|)-\frac{1}{a^2} \E[L^0_T].
    \end{align*}
Hence, we have 
\begin{align*}
\frac{1}{a^2} \E[L^0_T]\leq H_T(|X_0|).
\end{align*}
As in the proof of Lemma \ref{lem:Q geq V} we find that
\begin{align*}
    G(x,0,T) &=\sup_{X \in \mathcal C_{\mathcal A}, X_0=x}\limsup_{N\rightarrow \infty}\E \left[ \frac N2 \int_{0}^{T}1_{\left\{|X_s|\leq \frac{1}{N}\right\}}ds\right]\\
    &\leq H_T(|x|)
    \end{align*}
as required.
\end{proof}

Next, we complete the proof and show that $G=H$.
\begin{thm}[Main result]\label{t:main result 2}
	Let $G$ be given as in Problem \eqref{eq:OG} and $H$ as in Lemma \ref{H:explicit}. Then $G(x,y,T) = H_{T}(|x-y|)$ for any $T\geq 0$ and $x,y\in\mathbb R$.
\end{thm} 
\begin{proof}
Lemma \ref{lem:leq} yields that $G(x,y,T)\leq H_T(|x-y|)$ for any $T\geq 0$ and $x,y\in\mathbb R$. We assume without loss of generality that $y=0$ and need to show that
 $$ H_T(|x|) \leq G(x,0,T) $$
for any $x\in\mathbb R$, $T\geq 0$.

Fix $M\in\mathbb N_+$ and we construct a feedback control via choosing $$\sigma_M(x):= a+(b-a)g_M(|x|)$$ where $g_M\in C(\mathbb R_+,[0,1])$ such that for any $x\in\mathbb R$
\begin{equation*}
g_M(x)=
\begin{cases}
0 & \text{if} \quad |x|\in [0, \frac1 M],\\
1 & \text{if} \quad |x|\in [\frac2 M,\infty).\\
\end{cases}  
\end{equation*}
Let $X$ be the solution to the stochastic differential equation:
\begin{align}\label{X:conti}
dX_t = -k \sigma_M(X_t)^2 \mathrm{sign}(X_t) dt + \sigma_M(X_t)dW_t, \quad X_0=x, t\geq 0.
\end{align}
The control is now specified as:
 $$ \beta_t := -k \sigma_M(X_t)^2\mathrm{sign}(X_t), \quad \sigma_t := \sigma_M(X_t),\quad t\geq 0.$$
Note that $(\beta,\sigma)\in\mathcal A$ and therefore $X\in\mathcal C_{\mathcal A}$. As in the proof of Lemma \ref{lem:leq} we find that
\begin{align}\label{eq:H:exp}
\E[H_{T-t}(|X_t|)]&=H_T(|X_0|)\nonumber\\
	&\quad+\E\left[\int_{0}^{t}\Big(-\partial_{T} H_{T-s}(|X_s|)+ \sign(X_s)\partial_r H_{T-s}(|X_s|)\beta_s\right. \nonumber\\
 &\quad \left.+\frac{1}{2}\partial_{rr}H_{T-s}(|X_s|)\sigma^2_s\Big)1_{\{X_s\neq 0\}}ds\right]
-\E\left[\frac{1}{a^2}L^0_t\right]
\end{align}

With $\psi(t,x):=-\partial_{T} H_{T-t}(|x|)+ \sign(x)\partial_r H_{T-t}(|x|)\beta_t +\frac{1}{2}\partial_{rr}H_{T-t}(|x|)\sigma^2_t$, $x\in \mathbb R$, $0\leq t\leq T$, we have
\begin{align*}
\E\left[\int_0^T \psi(s,X_s)1_{\{X_s\neq0\}}ds \right] =\E\left[\int_0^T \psi(s,X_s)\left(1_{\left\{\frac{2}{M}\leq|X_s|\right\}}+1_{\left\{0<|X_s|< \frac{2}{M}\right\}}\right)ds \right]. 
\end{align*}

Since $H$ is the solution to Equation \eqref{eq:HJB1} and by using Lemma \ref{mono}, the supremum is attained with the drift and diffusion coefficient $(\beta, \sigma)$ given above when $|X_t|\geq \frac{2}{M}$. Thus, we have
\begin{align*}\E\left[\int_0^T \psi(s,X_s)1_{\left\{\frac{2}{M}\leq|X_s| \right\}}ds \right]=0.
\end{align*}
By Lemma \ref{epsilon bounds}, there exists some constant $C_3>0$ such that
\begin{align*}
\E\left[\int_0^T \left|\psi(s,X_s)\right| 1_{\left\{0<|X_s|< \frac{2}{M}\right\}}ds \right]
\leq \frac{C_3}{M^{\frac{1}{3}}} 
\end{align*}
where $C_3$ only depends on $a,b,k,T$ and it does not depend on the specific choice of $M$. Pushing $t$ to $T$ in Equation \eqref{eq:H:exp} yields
\begin{align*}
    H_T(|x|) &\leq \frac{1}{a^2}\E[L_T^0] + \frac{C_3}{M^{\frac13}}\\
    &= \frac1{a^2} 
\left(\limsup_{N \rightarrow \infty}   \E\left[\frac N2\int_0^T 1_{\left\{|X_s|\leq \frac 1 N\right\}}d\<X,X\>_s\right]\right)+\frac{C_3}{M^{\frac13}}\\
&\leq \sup_{Z \in \mathcal C_{\mathcal A}, Z_0=x}\left(\limsup_{N \rightarrow \infty}   \E\left[\frac N2\int_0^T 1_{\left\{|Z_s|\leq \frac 1 N\right\}}ds\right]\right) +  \frac{C_3}{M^{\frac13}}\\
&=G(x,0,T) +\frac{C_3}{M^{\frac13}}
\end{align*} 
where we used \cite[Corollary VI.1.9]{revuz2013continuous} for the first equality. This is true for all $M\in\mathbb N_+$ and, consequently, $H_T(|x|) \leq G(x,0,T)$ as required.
\end{proof}


\begin{rem}
    Assume that $dX_t = b(X_t)dt + \sigma(X_t)dW_t$ with $X_0\in\mathbb R$ where $b,\sigma\in C(\mathbb R,\mathbb R)$ with $\sigma(x)\in[a,b]$ and $|b(x)|\leq k |\sigma(x)|^2$ for any $x\in\mathbb R$ where $0<a\leq b$ and $k\geq 0$. Then \cite[Theorem VI.1.7, Corollary VI.1.6]{revuz2013continuous} yield that $X$ has a version of its local time $L$ which is continuous in time and space. The occupation times formula \cite[Corollary VI.1.6]{revuz2013continuous} yields that for any Borel set $A\subseteq \mathbb R$ one has
     $$ \int_0^t f_A(X_s)d\<X,X\>_s = \int_{-\infty}^\infty f_A(x) L_t^x dx$$
     where $f_A(x) := 1_A(x)/\sigma^2(x)$ for $x\in\mathbb R$. We see that
      $$ \int_0^t 1_A(X_s)ds = \int_{-\infty}^\infty 1_A(x) \frac{L_t^x}{\sigma^2(x)} dx = \int_A\frac{L_t^x}{\sigma^2(x)} dx . $$
     The fundamental theorem of calculus yields that
      $$ \lim_{N\rightarrow \infty} \frac{\int_{\left[y-\frac1N,y+\frac1N\right]}\frac{L_t^x}{\sigma^2(x)} dx}{2/N} = \frac{L_t^y}{\sigma^2(y)}$$
      for any $y\in\mathbb R$. Consequently, we have
       $$ \lim_{N\rightarrow \infty}\frac N2\int_0^t 1_{\left\{|X_s-y|\leq\frac1N\right\}} ds = \frac{L_t^y}{\sigma^2(y)}.$$
       In particular, $\frac{L_t^y}{\sigma^2(y)}$ is a version of the occupation density of $X$ at time $t$ in position $y$. We find that
      $$ \E\left[\lim_{N\rightarrow \infty}\frac N2\int_0^t 1_{\left\{|X_s-y|\leq\frac1N\right\}} ds\right] = \E\left[\frac{L_t^y}{\sigma^2(y)}\right]. $$
    Dominated convergence 
    yields
     $$ \lim_{N\rightarrow \infty}\E\left[\frac N2\int_0^t 1_{\left\{|X_s-y|\leq\frac1N\right\}} ds\right] = \E\left[\frac{L_t^y}{\sigma^2(y)}\right].$$
    On the other hand, applying  Theorem \ref{t:main result 2} yields
     $$  \E\left[\frac{L_t^y}{\sigma^2(y)}\right] = \lim_{N\rightarrow \infty}\E\left[\frac N2\int_0^t 1_{\left\{|X_s-y|\leq\frac1N\right\}} ds\right] \leq G(X_0,y,t).$$
\end{rem}

\appendix
\section{Computational bounds}
The following bound is technical and is needed in the proof of Theorem \ref{t:main result 2}.
\begin{lem}\label{epsilon bounds}
    Let $X$ be defined as the solution of the stochastic differential Equation \eqref{X:conti} and $H$ be defined as in Lemma \ref{H:explicit} Equation   \eqref{eq:H}. Let $\psi$ be defined as in the proof of Theorem \ref{t:main result 2}.
    
    Then we have for $M\in \mathbb N_+$ with $M>0$ and any fixed $T>0$
    \begin{align*}
        \E\left[\int_{0}^{T}\left|\psi(s,X_s)\right| 1_{\left\{0<|X_s|\leq \frac 2 M\right\}}ds\right] 
        \leq \frac{C_3}{M^{\frac13}} 
    \end{align*} 
    where $C_3 := 2\max\left(\frac{2^{\frac{5}{2}}b T^{\frac16}H^{\frac13}_T(0)}{a^2\sqrt{\pi}}, \frac{8b^2 k H_T(0)}{a^2}\right)$.
\end{lem}
\begin{proof}
We have
\begin{align}
\E&\left[\int_{0}^{T}\left|\psi(s,X_s)\right| 1_{\left\{0<|X_s|\leq \frac 2 M\right\}}ds\right] \nonumber\\
&=\E\bigg[\int_{0}^{T}\left|-\partial_{T} H_{T-s}(|X_s|)+ \sign(X_s)\partial_r H_{T-s}(|X_s|)\beta_s+\frac{1}{2}\partial_{rr}H_{T-s}(|X_s|)\sigma^2_s\right|\nonumber\\
&\qquad  \cdot 1_{\left\{0<|X_s|\leq \frac 2 M\right\}}ds\bigg]\nonumber\\
&\leq \int_{0}^{T}\sup_{|z|\in \left(0,\frac{2}{M}\right]}\left\{|-\partial_{T} H_{T-s}(z)|+|\partial_r H_{T-s}(z)|kb^2+\frac{1}{2}|\partial_{rr} H_{T-s}(z)|b^2  \right\}\nonumber\\
&\qquad \cdot \mathbb P\left(|X_s|\leq \frac 2 M\right)ds\label{A1}\\
&= \int_0^T \sup_{|z|\in \left(0,\frac{2}{M} \right]}\left|2\partial_{T} H_{T-s}(z) \right| \mathbb P\left(|X_s|\leq \frac 2 M\right)ds\label{A2}
\end{align}
where the last equality holds because $\partial_r H_t(z) \leq 0$, $\partial_{rr} H_t(z)\geq 0$ for any $t\geq 0$ and $|z|>0 $
 \begin{align}
    &\quad |-\partial_{T} H_{T-s}(z)|+|\partial_r H_{T-s}(z)|kb^2+\frac{1}{2}|\partial_{rr} H_{T-s}(z)|b^2 \nonumber\\
   &= 2|-\partial_{T} H_{T-s}(z)|\nonumber\\
   &\quad+\left\{-\partial_{T} H_{T-s}(z)+ \partial_r H_{T-s}(z)(-\mathrm{sign}(z)kb^2)+\frac{1}{2}\partial_{rr}H_{T-s}(z)b^2\right\}\label{AA1}\\
   &= 2|\partial_{T} H_{T-s}(z)|\nonumber
 \end{align}
 where the second equation holds since the sum in the bracket in Equation \eqref{AA1} is $0$ for any $z$ away from $0$.

Note that 
 $$ \partial_{T} H_{T-s}(z) = \frac{b}{a^2\sqrt{T-s}}\phi(v(z,T-s))+\frac{b^2 k}{a^2}\Phi(v(z,T-s)), \quad 0\leq z, 0\leq s\leq T$$
 with $v, \phi$ and $\Phi$ as in Lemma \ref{H:explicit}. 
 
Using the equation for $\partial_TH$ and the inequality from above we find
\begin{align}
\E&\left[\int_{0}^{T}\left|\psi(s,X_s)\right| 1_{\left\{0<|X_s|\leq \frac 2 M\right\}}ds\right] \nonumber \\
&\leq  \frac{\sqrt{2}b}{a^2\sqrt{\pi}}\int_0^T \left|\frac{1}{\sqrt{(T-s)}}\right| \mathbb P\left(|X_s|\leq \frac 2 M\right)ds +2\int_0^T \frac{b^2k}{a^2}\mathbb P\left(|X_s|\leq \frac 2 M\right)ds\nonumber
\\
&\leq  \frac{\sqrt{2}b}{a^2\sqrt{\pi}} \left\{\int_0^T s^{\left(-\frac{1}{2} \times \frac{3}{2}\right)} ds\right\}^{\frac{2}{3}} \left\{\int_0^T \left(\mathbb P\left(|X_s|\leq \frac 2 M\right)\right)^3 ds\right\}^{\frac{1}{3}}\nonumber\\
& \qquad+\frac{2b^2k}{a^2}\int_0^T \mathbb P\left(|X_s|\leq \frac 2 M\right)ds \nonumber 
\\
&\leq  \frac{2^{\frac{11}{6}}b T^{\frac16}}{a^2\sqrt{\pi}} \left\{\int_0^T \mathbb P\left(|X_s|\leq \frac 2 M\right)ds\right\}^{\frac{1}{3}}+\frac{2b^2k}{a^2}\int_0^T \mathbb P\left(|X_s|\leq \frac 2 M\right)ds \nonumber
\\
&\leq  \frac{2^{\frac{11}{6}}b T^{\frac16}}{a^2\sqrt{\pi}}\left(\frac{4H_T(0)}{M}\right)^{1/3}+\frac{8b^2 k H_T(0)}{a^2M} \nonumber\\
&= \frac{2^{\frac{5}{2}}b T^{\frac16}H^{\frac13}_T(0)}{a^2\sqrt{\pi}M^{\frac13}}+\frac{8b^2 k H_T(0)}{a^2M}\nonumber
\end{align}
where we bounded $\phi$ by $\frac{1}{\sqrt{2\pi}}$ and $\Phi$ with $1$ in the first inequality and we used H\"older's inequality for the second inequality. 
\end{proof}

\section{Application to integral bounds}
In order to demonstrate the usefulness of our main result Theorem \ref{thm:main} we apply it to obtain an integration bound for some expected time-integrals.
\begin{cor}\label{c:disintegration2}
    Assume the requirements of Corollary \ref{c:disintegration}. Further, let $T\geq 0$ and $f:[0,T]\times\mathbb R\rightarrow [0,\infty]$ be measurable and decreasing in the first variable, i.e.\ $f(t,x)\leq f(s,x)$ for any $s\leq t$ and $x\in\mathbb R$. Then
     \begin{align*}
         \E&\left[ \int_0^T f(s,X_s) ds \right] \leq \\
          &\int_{-\infty}^\infty \int_0^T \left(\frac{b}{a^2\sqrt{t}}\phi(v(|x-y|,t))+\frac{b^2 k}{a^2} \Phi(v(|x-y|,t))\right)f(t,y) dt dy
     \end{align*} 
where $v(r,t):= kb\sqrt{t}-\frac{r}{b\sqrt{t}}, 0\leq t, 0\leq r$, $\phi(z):=\frac{1}{\sqrt{2\pi}}e^{-\frac{z^2}{2}}$ and $\Phi(z):=\frac{1}{\sqrt{2\pi}}\int_{-\infty}^{z}e^{-\frac{s^2}{2}}ds$ for any $x,y\in \mathbb R$.
\end{cor}
\begin{proof}
    First assume that $f$ is continuously differentiable. Then one has    
    \begin{align*}
        f(s,x) &= f(T,x) + \int_s^{T} |\partial_t f(u,x)| du 
    \end{align*}
    Hence, we find that
     \begin{align*}
         \E\left[ \int_0^T f(s,X_s) ds \right] &= \E\left[ \int_0^T f(T,X_s) ds\right] + \E\left[ \int_0^T \int_s^T |\partial_t f(u,X_s)|du ds\right] \\
          &= \E\left[ \int_0^T f(T,X_s) ds\right] + \int_0^T \E\left[  \int_0^u |\partial_t f(u,X_s)|ds \right]du \\
          &\leq \int_{\mathbb R} f(T,y)G(x,y,T)dy + \int_0^T \int_{\mathbb R} |\partial_t f(u,y)| G(x,y,u) dy du \\
          &= \int_0^T \int_{\mathbb R} f(u,y) \partial_TG(x,y,u) dy du \\
          &= \int_{\mathbb R} \int_0^T f(u,y) \partial_T G(x,y,u) du dy
     \end{align*}
     as claimed if $f$ is continuously differentiable. Here, we used the equation for $f$ in the first equality, Tonelli's theorem for the second equality, Corollary \ref{c:disintegration} for the inequality, integration by parts applied to the outer right integral for the third equality and Tonelli's theorem for the final equality.

    By a monotone class argument we see that the inequality holds for all measurable functions $f$ which are decreasing in the first variable.
\end{proof}

\bibliographystyle{alpha}
\bibliography{main}

\begin{thebibliography}{ABSJ17}

\bibitem[ABSJ17]{ankirchner2017controlling}
Stefan Ankirchner, Christophette Blanchet-Scalliet, and Monique Jeanblanc.
\newblock Controlling the occupation time of an exponential martingale.
\newblock {\em Applied Mathematics \& Optimization}, 76:415--428, 2017.

\bibitem[AW21]{ankirchner2021sharp}
Stefan Ankirchner and Julian Wendt.
\newblock A sharp upper bound for the expected interval occupation time of
  {B}rownian martingales.
\newblock https://hal.science/hal-03138433/, 2021.

\bibitem[BK16]{banos2016}
David~R.\ Ba{\~n}os and Paul Kr{\"u}hner.
\newblock Optimal density bounds for marginals of {I}t{\^o} processes.
\newblock {\em Communications on Stochastic Analysis}, 10(2):131--150, 2016.

\bibitem[BS15]{borodin2015handbook}
Andrei~N.\ Borodin and Paavo Salminen.
\newblock {\em Handbook of {B}rownian motion -- facts and formulae}.
\newblock Springer Science \& Business Media, 2015.

\bibitem[EK23]{eisenbergmeasuring}
Julia Eisenberg and Paul Kr\"uhner.
\newblock Measuring the suboptimality of dividend controls in a {B}row\-nian
  risk model.
\newblock {\em Advances in Applied Probability}, pages 1--31, 2023.

\bibitem[FK81]{fabes1981examples}
Eugene~B.\ Fabes and Carlos~E.\ Kenig.
\newblock Examples of singular parabolic measures and singular transition
  probability densities.
\newblock {\em Duke Mathematical Journal}, 48(4):845--856, December 1981.

\bibitem[GH80]{geman1980occupation}
Donald Geman and Joseph Horowitz.
\newblock Occupation densities.
\newblock {\em The Annals of Probability}, pages 1--67, 1980.

\bibitem[IN94]{imkeller1994integration}
Peter Imkeller and David Nualart.
\newblock Integration by parts on {Wiener} space and the existence of
  occupation densities.
\newblock {\em The Annals of Probability}, pages 469--493, 1994.

\bibitem[KX23]{kruhner2023explicit}
Paul Krühner and Shijie Xu.
\newblock Explicit local density bounds for {I}t{\^o} processes with irregular
  drift.
\newblock https://arxiv.org/abs/2308.02241, 2023.

\bibitem[McN85]{mcnamara1985regularity}
John~M.\ McNamara.
\newblock A regularity condition on the transition probability measure of a
  diffusion process.
\newblock {\em Stochastics: An International Journal of Probability and
  Stochastic Processes}, 15(3):161--182, 1985.

\bibitem[RY13]{revuz2013continuous}
Daniel Revuz and Marc Yor.
\newblock {\em Continuous martingales and {B}rownian motion}, volume 293.
\newblock Springer Science \& Business Media, 2013.

\end{thebibliography}

\end{document}